%% file: cop10-arxiv.tex
\documentclass[11pt]{article}

\include{epic.sty}
\usepackage{amssymb,amsmath,amsthm,hyperref,srcltx,accents}
\usepackage{amsfonts,stmaryrd,mathrsfs}
\usepackage{paralist}
\usepackage{tikz}
\usepackage{subfigure}
\usepackage{times}
\usepackage{array}
\usepackage{multirow}
\usepackage{graphicx}

\usepackage{ifthen}
\newboolean{IncludeFragments}
\setboolean{IncludeFragments}{false}

\newlength{\originalbase}
\newcommand{\spacing}[1]{\setlength{\baselineskip}{#1\originalbase}}

\newif\ifnotesw\noteswtrue

\newtheorem{theorem}{Theorem}[section]

\newtheorem{lemma}[theorem]{Lemma}

\newtheorem{claim}[theorem]{Claim}
\newtheorem{observation}[theorem]{Observation}

\newtheorem{corollary}[theorem]{Corollary}
\newtheorem*{nonum}{}

\newcommand{\remove}[1]{}

\newcommand{\copwin}{\ \mathscr{C}}

\newcommand{\copmove}{{\rightarrow}}
\newcommand{\robmove}{{\rightarrow}}
\newcommand{\crmove}{{\twoheadrightarrow}}

\newcommand{\st}[2]{({#1}\,;\,{#2})}
\newcommand{\copst}[2]{(\underline{#1}\,;\,{#2})}
\newcommand{\robst}[2]{({#1}\,;\,\underline{#2})}

\newcommand{\bc}{C}

\begin{document}
\spacing{1.4}
\parskip=+3pt

\def\proofend{\hfill$\Box$\medskip}
\def\Proof{\noindent{\bf Proof. }}
\newcommand{\ProofOf}[1]{\noindent{\bf Proof of {#1}. }}

\newenvironment{proofof}[1]{\ProofOf{#1}}{\hfill $\Box$ \medskip}

\def\cN{\overline{N}}
\def\P{{\cal{P}}}
\def\cop{{\mathrm{cop}}}
\def\reg{{\mathrm{reg}}}

\title{The Petersen graph is the smallest 3-cop-win graph}

\author{Andrew Beveridge\thanks{Department of Mathematics, Statistics and Computer Science, Macalester College, Saint Paul MN 55105},\,\,
Paolo Codenotti\thanks{Institute for Mathematics and its Applications, University of Minnesota, Minneapolis MN 55455}, 
Aaron Maurer\thanks{Department of Mathematics, Carleton College, Northfield MN 55057}, 
 John McCauley\thanks{Department of Mathematics, Haverford College, Haverford PA 19041},
and Silviya Valeva\thanks{Department of Mathematics,  University of Iowa, Iowa City IA 52242} 
}

\maketitle

\begin{abstract}
In the game of \emph{cops and robbers} on a graph $G = (V,E)$, $k$
cops try to catch a robber.
On the cop turn, each cop may move to a neighboring vertex or remain
in place. On the robber's turn, he moves similarly.
The cops win if there is some time at which a cop is at the same
vertex as the robber. Otherwise, the robber wins.
The minimum number of cops required to catch the robber is called the
\emph{cop number} of $G$, and is denoted $c(G)$. 
Let $m_k$ be the minimum order of a connected graph satisfying $c(G) \geq k$. 
Recently, Baird and Bonato determined via computer search that $m_3=10$ and that this value is attained uniquely by the Petersen graph. Herein, we give a self-contained mathematical proof of this result. Along the way, we give some characterizations of  graphs with $c(G) >2$ and very high maximum degree.
\end{abstract}

\section{Introduction}

Let $G=(V,E)$ be a simple undirected graph on $n$ vertices. The game of cops and robbers on $G$ was independently introduced by Quilliot \cite{quilliot} and  Nowakowski and Winkler \cite{Nowakowski+Winkler}.
The game is played between $k$ cops $C_1, C_2, \ldots , C_k$ and one robber $R$.
First, the cops are placed at $k$ vertices of the graph. Then
the robber is placed on a vertex. During play, the cops and the robber move alternately. On the cop turn, each cop may move to a neighboring vertex or remain in place. The cops can coordinate their strategy and multiple cops may occupy the same vertex. On the robber's turn, he moves similarly. This is a full information game, in the sense that the locations of the cops and robber are always known to all players. 
The cops win if there is some finite time at which a cop is colocated with the robber. Otherwise, the robber wins.
The minimum number of cops required to catch the robber (regardless of robber's strategy) is called the \emph{cop number} of $G$, and is denoted $c(G)$. When $c(G)=k$, we say that $G$ is \emph{$k$-cop-win}.

The game of cops and robbers has received considerable attention in recent years; 
for an introduction see the surveys \cite{alspach, hahn} and the monograph \cite{bonato+nowakowski}.
The most important open question is Meyniel's conjecture that for every graph $G$ over $n$ vertices, $c(G)=O(\sqrt{n})$. The history of Meyniel's conjecture is surveyed in \cite{baird+bonato}, and the best known bound of
 $c(G) \leq n\, 2^{-(1+o(1))\sqrt{\log n}}$ was obtained independently in \cite{lu+peng,scott+sudakov,FKL}. Various authors have recently studied the cop number for random graph models
\cite{bollobas+kun,luczak+pralat,pralat,bonato+pralat+wang,bonato+kemes+pralat}. In particular,  
Pra{\l}at and Wormald \cite{pralat+wormald} have shown that Meyniel's conjecture holds for the Erd\H{o}s-Renyi random graph model.

The original papers \cite{quilliot, Nowakowski+Winkler} characterized the graphs for which $c(G)=1$. 
Given a vertex $v$, its  \emph{neighborhood} is 
$N(v) = \{ u \in V \mid (v,u) \in E \}$, and its \emph{closed neighborhood}  is
$\cN(v) = \{ v \} \cup N(v)$.  A vertex $v$ is \emph{dominated} by the vertex $w$ if $\overline{N}(v) \subseteq \overline{N}(w)$. A dominated vertex is also called a \emph{pitfall} or \emph{corner}.  A graph $G$ is \emph{dismantleable} if we can reduce $G$ to a single vertex by successively removing dominated vertices. These papers prove that 
a  connected graph $G$  has $c(G)=1$ if and only if $G$ is dismantleable.

Aigner and Fromme \cite{aigner+fromme} introduced the cop number as described above. In addition to proving that if $G$ is planar then $c(G) \leq 3$, they establish the following useful result. Let $\delta(G)$ denote the minimum degree of $G$. Recall that the \emph{girth} $g(G)$ of a graph is the size of its smallest cycle (if $G$ is acyclic then $g(G)$ is infinite). They proved that 
if $G$ is a graph with finite girth $g(G)\geq 5$ then $c(G) \geq \delta(G)$.
The Petersen graph $H$ is a 10 vertex, 3-regular graph of girth 5. Aigner and Fromme's result guarantees that the $c(H) \geq  3$, and it shows a winning 3-cop strategy.

Let $m_k$ denote the minimum order of a connected graph with $c(G) \geq k$.
Clearly, $m_1=1$ and $m_2=4$. Recently, Baird and Bonato \cite{baird+bonato}  used a computer search to prove that  $m_3=10$ and that this value is attained uniquely by the Petersen graph.
We give a self-contained mathematical proof of this result, split into two theorems.

\begin{theorem}\label{thm:9-vtx-c2}
If $G$ is a connected graph on at most $9$ vertices, then $c(G)\leq 2$.
\end{theorem}

\begin{theorem}
 \label{thm:petersen}
The Petersen graph is the unique connected graph on $10$ vertices that requires 3 cops. All other connected graphs of order $10$ are  2-cop-win.
\end{theorem}

Our proofs follow from  a  series of observations and lemmas about the cop number of graphs with very large maximum degree.  It is obvious that  a graph with a \emph{universal vertex} $v$ such that $\deg(v)=n-1$ is cop-win. We prove some structural results concerning graphs containing a vertex whose codegree is a small constant.

\begin{lemma}\label{lemma:gen-5-cycle}
Let $G$ be a connected graph on $n$ vertices. If there is a node $u\in
V(G)$ of degree at least $n-6$, then either $c(G)\leq 2$ or the induced
subgraph $G[V-\overline{N}(u)]$ is a 5-cycle.
\end{lemma}

\begin{corollary}
\label{cor:deg=n-5}
If $\Delta(G) \geq n-5$ then $c(G) \leq 2$. \endproof
\end{corollary}

Lemma \ref{lemma:gen-5-cycle} and its immediate corollary are crucial
tools in proving our main results. In particular,
Theorem~\ref{thm:9-vtx-c2} is a quick consequence of Corollary
\ref{cor:deg=n-5}. This reduces our search to 10 vertex graphs with $2
\leq \delta(G) \leq \Delta(G) \leq 4$. It takes some additional effort
to show that $c(G)=3$ forces $\Delta(G)=3$. At that point, we need the
following lemma.

\begin{lemma}
\label{lemma:n-7}
Let $G$ be a graph with a vertex $u$ with $\Delta(G) = \deg(u) = n-7$ and such that
$\deg(v) \leq 3$ for every $v \in V - \cN(u)$. Then either $c(G)\leq 2$ or the induced subgraph $G[V-\cN(u)]$ is a 6-cycle.
\end{lemma}

This lemma can be generalized a bit more. In particular, if we remove
the restriction on the maximum of degree of vertices in $V-\cN(u)$,
then the proofs of Lemmas~\ref{lemma:gen-5-cycle} and~\ref{lemma:n-7}
can be adapted to show that $H$ must contain an induced 5-cycle or
6-cycle. However, the case analysis is cumbersome, so we have opted
for this simpler formulation. The version stated above is sufficient
to prove our main result: that the Petersen graph is the only
10-vertex graph requiring 3 cops.

We conclude this section with some reflections on our main results. The
Petersen graph is the unique 3-regular graph of girth 5 of minimal
size, so that Theorem \ref{thm:petersen} provides a tight lower bound
for $n$ when $c(G)=3$. We wonder whether a similar result holds for
general cop numbers, and we formulate some open question in this
vein. Recall that a $(k,g)$-\emph{cage} is a $k$-regular graph with
girth $g$ of minimal order. For a survey of cages, see
\cite{cage}. The Petersen graph is the unique $(3,5)$-cage, and in
general, cages exist for any pair $k \geq 2$ and $g \geq 3$.  

As discussed earlier, Aigner and Fromme \cite{aigner+fromme} proved
that graphs with girth $5$, and degree $k$ have cop number at least
$k$; in particular, if $G$ is a $(k,5)$-cage then $c(G) \geq k$.  Let
$n(k,g)$ denote the order of a $(k,g)$-cage.  Is it true that a
$(k,5)$-cage is $k$-cop-win?  Next, since we have $m_k \geq
n(k,5)$, it is natural to wonder whether $m_k = n(k,5)$ for $k
\geq 4$. It seems reasonable to expect that this is true at least for
small values of $k$.  It is known that $n(4,5)=19$, $n(5,5)=30$,
$n(6,5)=40$ and $n(7,5)=50$. Do any of these cages attain the
analogous $m_k$?  

More generally, we can ask the same question for large $k$: is $m_k$ achieved by a $(k,5)$-cage?
If Meyniel's conjecture is true, then  $m_k= \Omega(k^2)$. Meanwhile, it is known that $n(k,5) = \Theta(k^2)$, so an affirmative resolution would be consistent with Meyniel's conjecture. We note that
Baird and Bonato
\cite{baird+bonato} have already observed  that Meyniel's conjecture implies that $m_k
= \Theta(k^2)$, using a projective plane construction to obtain the upper bound, rather than the existence of $(k,5)$-cages.


\section{Preliminaries}

\subsection{Definitions and Notation}

For a graph $G$, we will denote by $V(G)$ the vertex set of $G$, and
$E(G)$ the edge set of $G$. If the graph $G$ is clear from context, we
will sometimes use $V$ for $V(G)$ and $E$ for $E(G)$. We also use the
notation $v(G):=|V(G)|$. For $S\subseteq V$, the \emph{induced
  subgraph} on $S$, denoted $G[S]$, is the graph with vertex set $S$,
and edge set $E(G[S])=\{(u, v)\in E(G)\mid u, v\in S\}$.  
For sets $S, T$, we denote the difference of $T$ by $S$ by
$T-S :=\{t\in T\mid t\notin S\}$.

We specify some additional vertex notation. For $u,v \in V$, we write $u\sim
v$ when $(u, v)\in E$.  For $S \subset V$, we write $u \sim
S$ when $u \notin S$ and there exists $v \in S$ such that $u \sim
v$. We define $N(S)= \cup_{v \in S} N(v) - S$ and $\cN(S) = \cup_{v \in S} \cN(v)$. For convenience, we set
$N(u,v) = N(\{u,v\})$. 
For sets of vertices $S, T\subseteq V$, we denote the set of edges
between the two sets by $[S : T] := \{(s, t)\in E\mid s\in S, t\in
T\}$, and we use $|S:T|$ to denote the size of this edge set. We denote the \emph{degree} of a vertex $u$ by $\deg(u)$ and denote the minimum degree by 
$\delta(G) = \min_{v \in V(G)} \deg(v)$ and the maximum degree by $\Delta(G) = \max_{v \in V(G)} \deg(v)$.
We generalize the latter symbol to subsets of vertices: for $S\subseteq V$,
$\Delta(S) =\max_{s\in S} \deg(s)$.

We now introduce our notation for the state of the game. We fix a
connected graph $G$ on which the game is played. The state of the game
is a pair $\st{\bc}{r}$, where $G$ is a connected graph, $\bc$ is a
$k$-tuple of vertices $\bc = (c_1, c_2, \dots, c_k)$, where $c_i\in
V(G)$ is the current position of cop $C_i$, and $r\in V(G)$ is the
current position of the robber $R$.  For notational convenience, we
write $\st{c_1, \dots, c_k}{r}$ for $\st{(c_1, \dots, c_k)}{r}$. When
we need to specify whose turn it is to act, we underline the position
of the player whose turn it is: $\copst{\bc}{r}$ denotes that it is
the cops' turn to move, and $\robst{\bc}{r}$ the robber's.

We use a shorthand notation to describe moves: $\copst{c_1, \dots,
  c_k}{r}\copmove \robst{c'_1, \dots, c'_k}{r}$ denotes the cop
move where each $C_i$ moves from $c_i$ to $c'_i$. Similarly
$\robst{c_1, \dots, c_k}{r}\robmove \copst{c_1, \dots, c_k}{r'}$
denotes the robber's move from $r$ to $r'$. We will concatenate moves
and we use the shorthand $\crmove$, meaning a cop move followed by a
robber move:
\[\copst{c_1, \dots, c_k}{r}\crmove \copst{c'_1, \dots, c'_k}{r'} 
\equiv \copst{c_1, \dots, c_k}{r}\copmove
\robst{c'_1, \dots, c'_k}{r}\robmove\copst{c'_1, \dots,
  c'_k}{r'}. \] There will be cases where the strategy allows for
either the robber or the cops to be in one of several positions. In
general, for $T_i\subseteq V$, $S\subseteq V$, the state of the game
has the form $\st{T_1, \dots, T_k}{S}$ means that $c_i\in T_i$, and
$r\in S$.

The robber's \emph{safe neighborhood}, denoted $S(R)$, is the
connected component of $G-\cN(C)$
containing the robber. We say that the robber is \emph{trapped} when
$S(R) = \emptyset$. This condition is equivalent to having both $r \in
N(\bc)$ and $N(r) \subset \cN(\bc)$. Once the robber is trapped, he will
be caught on the subsequent cop move, regardless of the robber's next
action. When the robber is trapped, we are in  a \emph{cop-winning position}, denoted by
$\copwin$.

\subsection{The end game}

We frequently use the following facts to identify cop-win
strategies for two cops in the end game. 
We state a more general version of these results
for $k$ cops. 

We need the following property of a cop-win graphs, which first appears in \cite{clarke+nowakowski}. Every cop-win graph has at least one cop-winning \emph{no-backtrack strategy}, meaning that the cop never repeats a vertex during the pursuit. Typically, a graph has multiple no-backtrack strategies. We say that a vertex $v$ is \emph{no-backtrack-winning} if there is a cop-winning no-backtrack strategy starting at $v$. For example, when $G$ is a tree, every vertex is no-backtrack-winning.

Next we fix some notation. For a fixed set $U = \{ u_1, \dots, u_t \}
\subset V$, let $N_U'(u_j) := N(u_j) - \cN(U - u_{j} )$ be the
neighbors of $u_j$ that are not adjacent to any other vertex in $U$.

\begin{observation}\label{obs:key}
Let $\copst{\bc}{r}$ be the state of the game.  Suppose that  there exists a $c_j \in \bc$ such that either
\begin{inparaenum}[\upshape(\itshape a\upshape)]
\item \label{key:empty} $[S(R): N'_C(c_j)] = \emptyset$ and $G[S(R)]$ is cop-win; or
\item \label{key:one} 
$N(S(R)) \cap   N'_C(c_j) = \{ v \}$ such that $H=G[S(R) + v]$ is cop-win and $v$ 
is no-backtrack-winning in $H$;
  \end{inparaenum} 
  then the cops can win from this configuration.
\end{observation}

\begin{proof}
Let $S=S(R)$ be the initial safe neighborhood of $R$. In both cases, only cop $C_j$ is active, while the others remain stationary.   In case (a), cop $C_j$ moves into $S$ and follows a cop-win strategy on $G[S]$. 
In case (b),  cop $C_j$ moves  to $v$ and then follows a no-backtrack strategy on $G[S + v]$.  This prevents the robber from ever reaching $v$. In both cases, the only way for the robber to avoid capture by $C_j$ is to step into the neighborhood of the remaining cops.
\end{proof}

We highlight two useful consequences that are  used heavily for $k=2$ in our subsequent proofs.

\begin{corollary}\label{cor:small-safe-nbhd}
Let $\copst{\bc}{r}$ be the state of the game, played with $k \geq 2$
cops. If $|S(R)|\leq 2$ and $|N(S(R))|\leq 2k-1$, then the cops can
win.
\end{corollary}
\begin{proof}
Let $S=S(R)$.
  We have  $|N(S)\cap
  N(\bc)|\leq 2k-1$, so the pigeon hole principle ensures that  there exists a cop
  $C_j$ such that and $|N(S)\cap N'_C(c_j)|\leq 1$. We are done by Observation~\ref{obs:key}, since every vertex of a connected $2$-vertex graph is non-backtrack-winning.
%
\end{proof}

\begin{corollary}\label{cor:one-deg-3}
Let $\copst{\bc}{r}$ be the state of the game, played with $k \geq 2$ cops. If 
$\max_{v \in S(R)} \deg_G(v) \leq 3$ and $S(R)$ 
 contains at most one vertex of degree $3$, then the cops can win. \endproof
\end{corollary}

\begin{proof}
Let $S=S(R)$. Since $G[S]$ is connected, we have $[S : N(C)] \leq 3$. Therefore, some cop $C_j$ has $|S: N'_C(c_j)| \leq 1$. If $G[S]$ is a tree, then we are done by Observation~\ref{obs:key}.
If $G[S]$ is not a tree, then $G[S]$ must be unicyclic with one degree 3 vertex, say $u$. Therefore,  
$|S: N(C)|=1$, and except for $u$, every vertex in the cycle has degree 2 in $G$. A cop-winning strategy is as follows:  two cops move until they both reach $u$. Now $S(R)$ is a path, so  Observation~\ref{obs:key} completes the proof.
\end{proof}


\section{Graphs with $\Delta(G) \geq n-6$}

In this section, we prove Lemma \ref{lemma:gen-5-cycle} and Theorem \ref{thm:9-vtx-c2}. We also make progress on the proof of Theorem \ref{thm:petersen} by showing that if $v(G)=10$ and $\Delta(G)=4$, then $c(G)\leq 2$. For convenience, we recall the statements of these results prior to their respective proofs.

\begin{nonum}[{\bf Lemma \ref{lemma:gen-5-cycle}}]
Let $G$ be a connected graph on $n$ vertices. If there is a node $u\in
V(G)$ of degree at least $n-6$, then either $c(G)\leq 2$ or the induced
subgraph $G[V-\overline{N}(u)]$ is a 5-cycle. 
\end{nonum}

\begin{proof}
Let $H=G[V-\cN(u)]$. By Observation~\ref{obs:key}(\ref{key:empty}), if
$H$ is cop-win, then $c(G)\leq 2$. In particular this holds if $H$
does not contain an induced cycle of length at least 4. So we only
need to consider the case where $v(H) = 4$ or $5$, and $H$ contains an
induced $4$-cycle.  Let $x_1, x_2, x_3, x_4$ form the 4-cycle in $H$
(in that order). Let $x_5$ be the additional vertex (if present).

We now distinguish some cases based on $N(x_5)\cap H$. If $x_5\sim x_i$
for every $i\in\{1, 2, 3, 4\}$, then $H$ is cop-win, and hence
$c(G)\leq 2$. This leaves us with 5 cases to consider, depicted in
Figure~\ref{fig:4-cycle-cases}. Case (a) includes the situation when 
$\deg(u)=n-5$, and there is no vertex $x_5$.

\begin{figure}[ht]
\begin{center}
\begin{tabular}{ccccc}

\begin{tikzpicture}[scale=.85]

\input{fig-4-cycle-cases.tex}

\end{tikzpicture}

&
\begin{tikzpicture}[scale=.85]

\input{fig-4-cycle-cases.tex}

\draw (X2) -- (X5);

\end{tikzpicture}

&
\begin{tikzpicture}[scale=.85]

\input{fig-4-cycle-cases.tex}

\draw (X1) -- (X5);
\draw (X2) -- (X5);

\end{tikzpicture}
&

\begin{tikzpicture}[scale=.85]

\input{fig-4-cycle-cases.tex}

\draw (X2) -- (X4);

\end{tikzpicture}
&
\begin{tikzpicture}[scale=.85]

\input{fig-4-cycle-cases.tex}

\draw (X1) -- (X5);
\draw (X2) -- (X5);
\draw (X4) -- (X5);

\end{tikzpicture}

\\
(a) & (b) & (c) & (d) & (e) 

\end{tabular}

\caption{The five cases for $G[V - \cN(u)]$ in  Lemma~\ref{lemma:gen-5-cycle}.}

\label{fig:4-cycle-cases}
\end{center}
\end{figure}
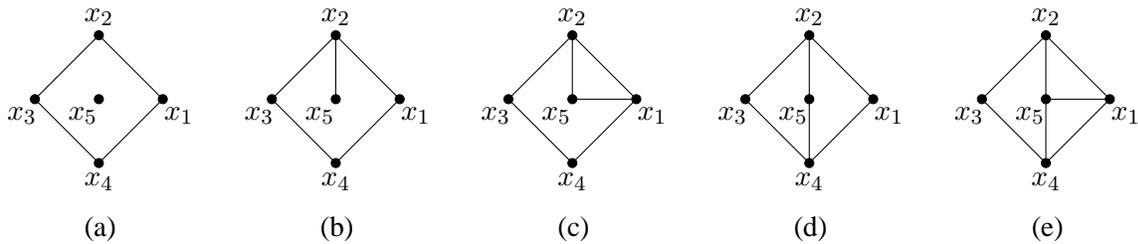

First we make some technical observations. We start by noting that
moving to $x_5$ is in most situations a bad idea for the robber in
Cases (a), (b) and (c).

\begin{claim}\label{claim:Rx5bad}
In Cases (a), (b), and (c), if the state of the game is of the form 
$\copst{\cN(u), V(H)}{x_5}$, the cops have a winning strategy.
\end{claim}
\begin{proof}
$C_1$ moves to $u$. In Case (a), $S(R) = \{x_5\}$ and we are already
  done by Corollary~\ref{cor:small-safe-nbhd}. In Case (b), if possible,
  $C_2$ moves directly to $x_2$; otherwise $C_2$ moves first to $x_1$
  and then to $x_2$; in either case the robber is trapped at $x_5$. In
  Case (c), $C_2$ moves to $x_2$ or $x_1$ (whichever $c_2$ is adjacent
  to), again trapping the robber in $x_5$.
\end{proof}

Next we look at the structure of $N(y)\cap V(H)$ for nodes $y\in
N(u)$.

\begin{claim}\label{claim:good-state}
Suppose the state of the game has the form $\copst{\cN(u), \{x_1,
  x_3\}}{y}$, where $y\in N(u)$ is such that
either \begin{inparaenum}[\upshape(\itshape a\upshape)]
\item \label{case:both} $N(y)\cap V(H)
  = \{x_2, x_4\}$, or \item \label{case:one}$y$ is adjacent to at most one of $x_2$ or
  $x_4$; \end{inparaenum} then the cops have a winning strategy.
\end{claim}

\begin{figure}
\begin{center}
\begin{tabular}{cccc}
\begin{tikzpicture}[scale=.85]

\input{fig-4-cycle2.tex}

\draw (Y) -- (X2);
\draw (Y) to [bend right] (X4);

\end{tikzpicture}

&
\begin{tikzpicture}[scale=.85]

\input{fig-4-cycle2.tex}

\draw (Y) -- (X2);
\draw[dashed] (Y) to [bend right] (X5);
\draw[dashed] (X1) to [bend right] (Y);
\draw[dashed] (X3) to [bend left] (Y);

\end{tikzpicture}
&

\begin{tikzpicture}[scale=.85]

\input{fig-4-cycle2.tex}

\draw (Y) to [bend right] (X4);
\draw[dashed] (Y) to  (X5);
\draw[dashed] (X1) to [bend right] (Y);
\draw[dashed] (X3) to [bend left] (Y);

\end{tikzpicture}
&

\begin{tikzpicture}[scale=.85]

\input{fig-4-cycle2.tex}

\draw[dashed] (Y) to  (X5);
\draw[dashed] (X1) to [bend right] (Y);
\draw[dashed] (X3) to [bend left] (Y);

\end{tikzpicture}
\\
(A) & (B1) & (B2) & (B3)
\end{tabular}

\caption{The four classes of possible structures of $G$ for Claim \ref{claim:good-state}. Vertex $x_5$  might not be present, and dashed edges might not be present.}
\label{fig:4cycle-continued}

\end{center}
\end{figure}
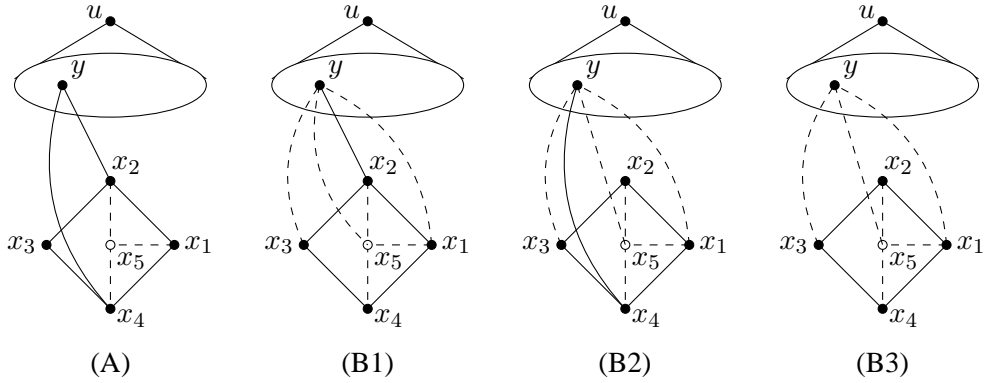

\begin{proof}
Figure \ref{fig:4cycle-continued} shows the four classes of possible graph structures. 
Let us first consider the structure (B1). Let $z=x_2$.
$C_1$ moves to $u$, and $C_2$ moves
to $z$. Now the robber is trapped in all cases of Figure \ref{fig:4-cycle-cases} except Case (a). In Case (a) the robber's only move is to $x_5$. After this move,  the robber can be caught by Claim~\ref{claim:Rx5bad}.
The same cop strategy works for structures (B2) and (B3), taking $z=x_4$. A simplified version of this proof shows that the same cop strategy works for structure (A), taking $z=c_2$.
\end{proof}

\noindent\textsc{Remark:} in Cases (d) and (e) of Figure \ref{fig:4-cycle-cases}, $x_5$ and $x_1$ are
symmetric, so the statement holds also for configuration $\copst{\cN(u), x_5}{y}$

Our next observation concerns the situation where there are two nodes in $N(u)$
that do not satisfy the condition of the previous claim.

\begin{claim}\label{claim:2-dominating}
If there are two vertices $y, z\in N(u)$ such that $ \{ x_2, x_3,x_4 \} \subset N(y)$, and 
$\{ x_1, x_2, x_4 \} \subset N(z)$, then $c(G)\leq 2$.
\end{claim}
\begin{proof}
First we deal with all cases but Case (d). The cops start at $u$ and
$z$. If the robber starts at $x_3$, the cops' winning strategy is:
$\copst{u, z}{x_3}\copmove \robst{u, y}{x_3}\copwin$.  If the robber
starts at $x_5$, the strategy will depend on the structure of $H$. In
Cases (a),(b),(c), we are done by Claim~\ref{claim:Rx5bad}. In Case
(e), the following is a winning strategy: $\copst{u, z}{x_5}\copmove
\robst{u, x_1}{x_5}\copwin$.

The remainder of the proof deals with Case (d), which requires a more
involved argument.

First suppose that there exists $w\in N(u)$ such that $\{x_2, x_4, x_5\}\subset N(w)$. Then the cops start at $u$ and $z$. The robber can
start at $x_3$ or $x_5$ in either case the cops have a winning
strategy: $\copst{u, z}{x_3}\copmove \robst{y,
  u}{x_3}\copwin$; or $\copst{u, z}{x_5}\copmove \robst{w,
  u}{x_5}\copwin$.

Now assume that no such $w$ exists. Start the cops at $u$ and $y$. The
robber starts in  $\{ x_1, x_5\}$. If the robber starts at $x_1$,
then $\copst{u, y}{x_1}\copmove\robst{y, u}{x_1}\copwin$. So we may assume the
robber starts at $x_5$. If $|N(x_5)\cap N(u)|\leq 1$, we are done by
Corollary~\ref{cor:small-safe-nbhd}. Otherwise, the cops move by $\copst{u,
x_3}{x_5}\copmove \robst{v, z}{x_5}$, for some $v\in N(x_5)\cap N(u)$. The
robber is forced to move to some $w\in N(x_5)\cap N(u)$ (if no such
$w$ exists, then $R$ is trapped). By our initial argument, $w$ cannot
be adjacent to both $x_2$ and $x_4$, but then the state satisfies
the conditions of Claim~\ref{claim:good-state}(\ref{case:one}), and we are done.
\end{proof}

%

\begin{claim}\label{claim:x1-N(u)}
Either $c(G)\leq 2$, or we can relabel the vertices of $H$ via an
automorphism of $H$ so that $x_1$ is adjacent to $N(u)$. 
\end{claim}
\begin{proof} Suppose that no such relabeling exists. We will show a
winning strategy for the cops, starting at $u$ and $x_3$. In Cases (a)
and (b) the claim follows from Corollary~\ref{cor:small-safe-nbhd}
(either $S(R) = \{x_1\}$ or $S(R) = \{x_5\}$). In Cases (d) and (e), $S(R)\subseteq\{x_1, x_5\}$,
and we are assuming that both $x_1$ and $x_5$ have no edges to $N(u)$;
hence $|N(S(R))|\leq 2$, and we are again done by
Corollary~\ref{cor:small-safe-nbhd}. In Case (c) $S(R)=\{x_1, x_5\}$,
and we are assuming that both $x_1$ and $x_2$ do not have neighbors in
$N(u)$. By Claim~\ref{claim:Rx5bad}, we may assume $R$ does not start
at $x_5$, and so $R$ starts at $x_1$. Let $v\in N(u)\cap N(x_5)$ (if
$x_5\nsim N(u)$, then $N(S(R))$ is dominated by $c_2=x_3$). Now the
cops can win by following the strategy: $\copst{u, x_3}{x_1}\copmove
\copst{v, x_3}{x_1}\copwin$.
\end{proof}

Armed with the above claims, we now conclude the proof
Lemma~\ref{lemma:gen-5-cycle}.

By Claim~\ref{claim:x1-N(u)}, we may assume $x_1\sim w\in N(u)$. Initially
place $C_1$ at $u$ and $C_2$ at $x_1$. The robber could start at $x_3$
or, in Cases (a), (b), and (d), at $x_5$. If the robber starts at
$x_5$ in Cases (a) and (b), then we are done by
Claim~\ref{claim:Rx5bad}. In Case (d), $x_5$ and $x_3$ are symmetric,
so without loss of generality, $r = x_3$, and the initial state is
$\copst{u, x_1}{x_3}$.

If $x_3\nsim N(u)$, then the cops win by Corollary~\ref{cor:small-safe-nbhd}.
Otherwise let $v\in N(x_3)\cap N(u)$. Then $C_1$ moves from $u$ to
$v$, while $C_2$ remains fixed at $x_1$, forcing $R$ to some $y\in
N(u)\cap N(x_3)$, with $y\nsim v, y\nsim x_1$.  If no such $y$ exists, then $R$ is
trapped. If $y$ is adjacent to only one of $x_2$ or $x_4$, we are in
the state $\copst{v\in N(u), x_1}{y}$, which satisfies the conditions of
Claim~\ref{claim:good-state}(\ref{case:one}), and hence the cops have
a winning strategy.

Otherwise $y$ is adjacent to $x_2, x_3$, and $x_4$. The cops move
$\copst{v, x_1}{y}\copmove \robst{x_3, w}{y}$, for some $w\in
N(x_1)\cap N(u)$. If $y\sim x_5$, and $R$ moves to $x_5$, then the
cops win: in Cases (a),(b),(c) we are done by
Claim~\ref{claim:Rx5bad}; in Case (d), $\copst{x_3, w}{x_5}\copmove
\robst{y, u}{x_5}\copwin$; in Case (e), the cops can adopt a different
strategy from the beginning: $\copst{u, y}{x_1}\copmove \robst{u,
  x_5}{x_1}\copwin$. The only other option is for $R$ to move to some
$z\in N(u)$, $z\nsim x_3$. So the state is $\copst{x_3, w}{z}$. Either
the pair $y, z$ satisfies the conditions of
Claim~\ref{claim:2-dominating}, or the current state satisfies the
conditions of Claim~\ref{claim:good-state}(\ref{case:one}) or
(\ref{case:both}).  In either case, we are done. This concludes the
proof of Lemma \ref{lemma:gen-5-cycle}.
\end{proof}

We now state some quick but useful consequences of
Lemma~\ref{lemma:gen-5-cycle}.

\begin{corollary}
\label{cor:deg4b}
Let $G$ be a graph on $n$ vertices. If there is a vertex $u\in V$ of
degree at least $n-6$, and a vertex $v\in V-\cN(u)$ such that
$|N(v)-N(u)|\geq 3$, then $c(G) \leq 2$.  \endproof
\end{corollary}
\begin{proof}
$v$ has three neighbors in $G[V-\cN(u)]$, and hence $G[V-\cN(u)]$ cannot be a 5-cycle.
\end{proof}

\begin{corollary}~\label{cor:deg4deg3}
Let $G$ be a graph on $n$ vertices. If there is a vertex $u$ of degree
at least $n-6$ and a vertex  $v \in V -\cN(u)$ with $\deg(v) \leq 3$, then
$c(G)\leq 2$.
\end{corollary}

\begin{proof}
By Lemma~\ref{lemma:gen-5-cycle}, we only need to consider the case
where $G[V - \cN(u)]$ is a 5-cycle, $x_1, x_2, x_3, x_4, x_5$ (in that order). Without
loss of generality, let $\deg(x_1)\leq 3$, and $\deg(x_2)\geq 3$. For
each $i=1, \dots, 5$ such that $\deg(x_i)\geq 3$, pick some $y_i\in
N(x_i)\cap N(u)$ arbitrarily (we allow $y_i=y_j$ for $i\neq
j$). The game starts as $\copst{u,x_4}{\{x_1,x_2\}} \crmove \copst{u, \{x_3, x_4\}}{x_1}$.
First we deal with the case where $\deg(x_1)=2$ and the case where 
$\deg(x_1)=3$ and $y_1\sim x_4$. 
The cops' winning strategy for these two cases is the same: $\copst{u,
  \{x_3, x_4\}}{x_1}\crmove \copst{y_2, x_4}{x_1}\copmove \robst{x_2,
  x_4}{x_1}\copwin.$

Now we may assume that all $x_i$ have degree $3$, and hence $y_i$ exists for
all $i$. We may further assume that $x_4\neq y_1$, and, since $x_3$ and $x_4$ are symmetric, we are also done in the case $y_1\sim x_3$. 
The only remaining possibility is $N(y_1)\cap
(V - \cN(u))\subseteq \cN(x_1)$.  Since $x_3$ and $x_4$ are symmetric,
without loss of generality, the state is $\copst{u, x_4}{x_1}$. The
cops first move to $y_2$ and $x_5$, forcing the robber to $y_1$, then
in one more move, the robber is trapped at $y_1$: $\copst{y_2,
  x_5}{y_1}\copmove\robst{u, x_1}{y_1}\copwin.$
\end{proof}

These corollaries are enough to prove that every connected 9-vertex
graphs is 2 cop-win, and to show that if $v(G)=10$ and $\Delta(G)=4$
then $c(G) \leq 2$.

\begin{nonum}[{\bf Theorem \ref{thm:9-vtx-c2}}]
If $G$ is a connected graph on at most $9$ vertices, then $c(G)\leq 2$.
\end{nonum}

\begin{proof}
If $\Delta(G)\geq 4$, then we are done by
Lemma~\ref{lemma:gen-5-cycle}. If $\Delta(G) = 3$, then we are done by
Corollary~\ref{cor:deg4deg3}.
\end{proof}

\begin{lemma}\label{lemma:deg-leq-3}
If $v(G)=10$ and $\Delta(G)\geq 4$, then $c(G)\leq 2$.
\end{lemma}
\begin{proof}
Let $u\in V(G)$ have degree at least $4$. By
Lemma~\ref{lemma:gen-5-cycle}, either $c(G)\leq 2$ or $\deg(u)=4$, and
$G[V - \cN(u)]$ is a 5-cycle. Now, by Corollary~\ref{cor:deg4deg3},
either $c(G)\leq 2$, or every $u\in V - \cN(u)$ has $\deg(u)\geq 4$. In
the latter case, $|[N(u):V-\cN(u)]|\geq 10$, thus, by the pigeon hole
principle, there exists $v\in N(u)$ such that $|N(v)\cap (V- \cN(u))|\geq
3$. We now deal with this case, namely $u$ and $v$ have degree $4$,
and $N(u)\cap N(v) = \emptyset$.

By Lemma~\ref{lemma:gen-5-cycle}, both $G[V(G)-\cN(u)]$ and
$G[V(G)-\cN(v)]$ are $5$-cycles. The resulting graph
structure must be one of the two shown in Figure \ref{fig:deg4c}.
Considering the structure in Figure \ref{fig:deg4c}(a), we note that
$\deg(z_1) = \deg(z_2)=3$ in order to maintain the induced 5-cycle
structures, and hence we are done by Corollary~\ref{cor:deg4deg3}.

Now suppose that $G$ has the structure in Figure
\ref{fig:deg4c}(b). In this case we show $\deg(x_3)=3$, and we are
again done by Corollary~\ref{cor:deg4deg3}. To show that
$\deg(x_3)=3$, we look at each potential additonal edge, and show that
$V - \cN(x_3)$ is not a 5-cycle, and hence we are done by
Lemma~\ref{lemma:gen-5-cycle}. We only need to consider edges to $y_1,
y_2$ or $y_3$: other potential edges would not maintain the induced
5-cycle structure. We have $x_3 \nsim y_1$ because $\{v, y_2, y_3 \}$
form a triangle. We have $x_3 \nsim y_3$ because $z_1$ is adjacent to
each of $x_1, y_1, y_2$. Finally, $x_3 \nsim y_2$ because the
existence of this edge would force $y_3 \sim x_1$, which is symmetric
to the forbidden $x_3 \sim y_1$.
\end{proof}

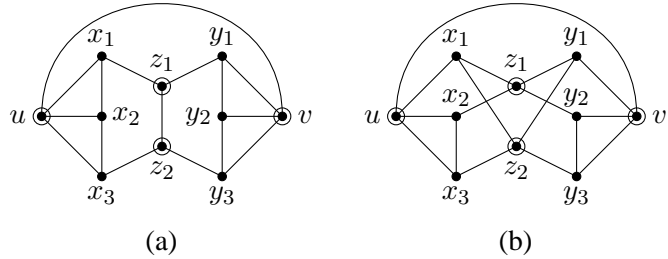
\begin{figure}[ht]
\begin{center}
\begin{tabular}{cc}

\begin{tikzpicture}[scale=.8]

\path (0,0) coordinate (U);
\path (4,0) coordinate (V);

\path (2,-.5) coordinate (Z2);
\path (2,.5) coordinate (Z1);

\path (1,-1) coordinate (X3);
\path (1,0) coordinate (X2);
\path (1,1) coordinate (X1);

\path (3,-1) coordinate(Y3);
\path (3,0) coordinate (Y2);
\path (3, 1) coordinate (Y1);

\draw[fill] (U) circle (2pt);
\draw[fill] (V) circle (2pt);

\draw (U) circle (4pt);
\draw (V) circle (4pt);

\node[left=2pt] at (U) {$u$};
\node[right=2pt] at (V) {$v$};

\foreach \i in {1,2,3}
{
\draw[fill] (Y\i) circle (2pt);
\draw (V) -- (Y\i);

\draw[fill] (X\i) circle (2pt);
\draw (U) -- (X\i);
}

\node[above] at (X1) {$x_1$};
\node[above] at (Y1) {$y_1$};

\node[right] at (X2) {$x_2$};
\node[left] at (Y2) {$y_2$};

\node[below] at (X3) {$x_3$};
\node[below] at (Y3) {$y_3$};

\foreach \i in {1, 2}
{
\draw[fill] (Z\i) circle (2pt);
\draw (Z\i) circle (4pt);
}

\node[above=2pt] at (Z1) {$z_1$};
\node[below=2pt] at (Z2) {$z_2$};
\draw (Z1) -- (X1) -- (X2) -- (X3) -- (Z2) -- (Y3) -- (Y2) -- (Y1) -- (Z1) -- (Z2);

\path (0,2.5) coordinate (CU);
\path (4,2.5) coordinate (CV);
\draw (U) .. controls (CU) and (CV) .. (V);

\end{tikzpicture}

&
\begin{tikzpicture}[scale=.8]

\path (0,0) coordinate (U);
\path (4,0) coordinate (V);

\path (2,-.5) coordinate (Z2);
\path (2,.5) coordinate (Z1);

\path (1,-1) coordinate (X3);
\path (1,0) coordinate (X2);
\path (1,1) coordinate (X1);

\path (3,-1) coordinate(Y3);
\path (3,0) coordinate (Y2);
\path (3, 1) coordinate (Y1);

\draw[fill] (U) circle (2pt);
\draw[fill] (V) circle (2pt);

\draw (U) circle (4pt);
\draw (V) circle (4pt);

\node[left=2pt] at (U) {$u$};
\node[right=2pt] at (V) {$v$};

\foreach \i in {1,2,3}
{
\draw[fill] (Y\i) circle (2pt);
\draw (V) -- (Y\i);

\draw[fill] (X\i) circle (2pt);
\draw (U) -- (X\i);
}

\node[above] at (X1) {$x_1$};
\node[above] at (Y1) {$y_1$};

\node[above] at (X2) {$x_2$};
\node[above] at (Y2) {$y_2$};

\node[below] at (X3) {$x_3$};
\node[below] at (Y3) {$y_3$};

\foreach \i in {1, 2}
{
\draw[fill] (Z\i) circle (2pt);
\draw (Z\i) circle (4pt);
}

\node[above=2pt] at (Z1) {$z_1$};
\node[below=2pt] at (Z2) {$z_2$};
\draw (Z1) -- (X1) -- (Z2);
\draw (Z1) -- (X2) -- (X3) -- (Z2);
\draw (Z1) -- (Y1) -- (Z2);
\draw (Z1) -- (Y2) -- (Y3) -- (Z2);

\path (0,2.5) coordinate (CU);
\path (4,2.5) coordinate (CV);
\draw (U) .. controls (CU) and (CV) .. (V);

\end{tikzpicture} \\
(a) & (b) 
\end{tabular}

\caption{The two possible starting structures in the proof of Lemma
  \ref{lemma:deg-leq-3}. Circled vertices cannot have additional edges.}

\label{fig:deg4c}

\end{center}
\end{figure}

\section{Graphs with $\Delta(G) = n-7$}

In this section, we prove Lemma \ref{lemma:n-7} and complete the proof of Theorem \ref{thm:petersen}.

\begin{nonum}[{\bf Lemma \ref{lemma:n-7}}]
Let $G$ be a graph with a vertex $u$ with $\Delta(G) = \deg(u) = n-7$ and such that
$\deg(v) \leq 3$ for every $v \in V - \cN(u)$. Then either $c(G)=2$ or the induced subgraph $G[V-\cN(u)]$ is a 6-cycle.
\end{nonum}

\begin{proof}
Let $H=G[V-\cN(u)]$ and suppose that $c(G) >2$. First, we observe that $H$ must be connected. Otherwise, we can adapt the proof of Corollary~\ref{cor:deg4deg3} to show that $c(G)=2$. Indeed, $H$ has at most one component $H_1$ whose cop number is 2. We use the strategy described in the proof of Corollary~\ref{cor:deg4deg3} to capture the robber. The only alteration of the strategy is to address the robber moving from $N(u)$ to $H - H_1$. However, $|V(H-H_1)| \leq 2$, so this component is cop-win. One cop responds by moving to $u$, while the other moves into $H-H_1$ for the win (Observation~\ref{obs:key}(a)).

Therefore, we may assume that $H$ is connected and  $c(H) \geq 2$. This means that $H$ must contain an induced $k$-cycle for $k \in {4,5,6}.$
Suppose that $G$ contains an induced 4-cycle $x_1, x_2, x_3, x_4$. Without loss of generality, $x_5 \sim x_1$, and $x_6$ is adjacent to at most three of $\{ x_2,x_3,x_4, x_5 \}$ (because we already have $\deg(x_1) =3$). Start the cops at $u$ and $x_1$, so that $S(R)$ is one of $\{ x_3 \}$, $\{ x_6 \}$ or $\{x_3, x_6\}$. In the first two cases, $\Delta(S(R)) \leq 3$ so the cops win by Corollary \ref{cor:one-deg-3}. The last option occurs when $x_3 \sim x_6$. If $x_6$ has at most one neighbor in $N(u)$, we are again done by Corollary~\ref{cor:one-deg-3}, since $\Delta(S(R)) \leq 3$. When $x_6$ has two neighbors in $N(u)$, the game play depends on the initial location of the robber. If the robber starts at $x_6$ then $C_1$ holds at $u$ while $C_2$ moves from $x_1$ to $x_2$ to $x_3$, trapping the robber. If the robber starts at $x_3$, then the roles are reversed: $C_1$ moves to $x_6$ in two steps while $C_2$ holds at $x_1$. At this point, the robber is trapped.

Next, suppose that $G$ contains an induced 5-cycle $x_1, x_2, x_3, x_4, x_5$. Without loss of generality, $x_6 \sim x_1$. If $x_6$ is adjacent to two of the $x_i$, then we can place $C_1$ at $u$ and $C_2$ at some $x_j$ so that $|N(S(R))\cap N(u)| \leq 1$, giving a cop winning position by Observation~\ref{obs:key}(b). Indeed, by symmetry there are only 2 cases to consider: if $x_6\sim x_2$, then $C_2$ starts at $x_4$ and $S(R) = \{x_1, x_2, x_6\}$; if $x_6\sim x_3$, then $C_2$ starts $x_3$, and $S(R) =\{x_1, x_5\}$.
So we may assume that $x_6$ has no additional neighbors in $H$. There are two cases to consider.
If $x_2$ and $x_4$ do not share a neighbor in $N(u)$, then our game play begins with $C_2$ chasing $R$ onto $x_2$: 
$\copst{u,x_1}{\{x_3, x_4\}} \crmove \cdots  \crmove  \copst{u, \{x_4, x_5\}}{x_2}$. If $x_2$ is not adjacent to $N(u)$, then the cops can ensure $S(R)$ satisfies Corollary \ref{cor:one-deg-3} on their next move. Indeed, $C_2$ moves to $x_4$. If $N(x_6)\cap N(u)=\emptyset$, then the situation already satisfies Corollary~\ref{cor:one-deg-3}, otherwise $C_1$ moves to $N(x_6)\cap N(u)$, and now the situation satisfies Corollary~\ref{cor:one-deg-3}. 

The final case to consider is when $x_2$ and $x_4$ are both adjacent to $y \in N(u)$. By symmetry, $x_3$ and $x_5$ are adjacent to $z \in N(u)$. By symmetry, there is one game to consider:
$\copst{u,x_1}{x_3} \crmove \copst{z,x_2}{x_4}$ which is cop-win by Corollary \ref{cor:one-deg-3}.
Thus, the only option for $H$ is an induced 6 cycle.
\end{proof}

We can now prove that the Petersen graph is the unique 3 cop-win graph of order 10.

\begin{observation}\label{fact:petersen}
The Petersen graph is the only $3$-regular graph $G$ such that for every vertex $u\in V(G)$, $G[V(G)-\cN(u)]$ is a $6$-cycle.
\end{observation}
\begin{proof}
This is easily
checked against the 18 possible 3-regular graphs of order 10 listed at
\cite{mckay}, but here we give a direct proof.
 Pick any vertex $u$. The
  complement is a 6-cycle, where every vertex is adjacent to exactly
  one vertex in $N(u)$. Let $N(u)= \{y, z, w\}$. Now pick a vertex
  $x_1$ on the 6-cycle, $x_1\sim y$. Because $V-N(x_1)$ must form a
  6-cycle, we must have that $x_3\sim w$ and $x_5\sim z$ (by symmery
  this is the only option). The only remaining edges to add are a
  matching between $x_2, x_4, x_6$ and $y, z, w$. To avoid a triangle
  in $V-N(y)$, we cannot have $x_4\sim y$ or $x_4\sim z$, hence
  $x_4\sim y$, and, by symmetry, $x_2\sim z$, and $x_6\sim w$. This is the
  Petersen graph. 
\end{proof}

\begin{nonum}[{\bf Theorem \ref{thm:petersen}}]
The Petersen graph is the unique connected graph on $10$ vertices that requires 3 cops. All other connected graphs of order $10$ are  2-cop-win.
\end{nonum}

\begin{proof}
Let $G$ be a connected graph of order 10 such that $c(G)=3$. 
We have $\delta(G) \geq 2$: otherwise the leaf $v \in V(G)$ is a dominated vertex, so $c(G)=c(G-v) \leq 2$ by Theorem
\ref{thm:9-vtx-c2}. Lemma \ref{lemma:deg-leq-3} ensures that $\Delta(G) \leq 3$. Clearly $\Delta(G)=3$ since a connected 2-regular graph is a cycle which is 2-cop-win.
  
  Suppose a vertex $u \in V(G)$ has
$\deg(u)=3$. Then, by Lemma~\ref{lemma:n-7}, $G[V - \cN(u)]$ must be a
6-cycle. If every vertex in $N(u)$ has degree 3, then $G$ is 3-regular
with $c(G)=3$, and therefore $G$ is the Petersen graph by Fact~\ref{fact:petersen}. Otherwise
there is a vertex $x_1 \in V - \cN(u)$ with $\deg(v)=2$. In the rest of
the proof we give a winning strategy for the cops in this case.

Let the 6-cycle $G[V - \cN(u)]$ be $\{ x_1, x_2, x_3, x_4, x_5, x_6 \}$. Without loss of generality, $\deg(x_1)=2$ and $\deg(x_2)=3$. Let $k=\max\{i\mid \deg(x_i)=3\}$. The initial configuration is
$\copst{u,x_4}{\{x_1,x_2,x_6\}}$. If $k \leq 5$ then the cops win by Corollary \ref{cor:one-deg-3}.
When $k=6$, our strategy depends on the initial robber location. Let $y \in N(u) \cap N(x_2)$. We either have
$\copst{u,x_4}{x_2} \crmove \copst{y,x_4}{x_1} \copmove \robst{y,x_5}{x_1}\copwin$, or
$\copst{u,x_4}{x_1} \crmove \copst{y,x_5}{x_1}\copwin$, or
$\copst{u,x_4}{x_6} \crmove \copst{u,x_5}{x_1}\ \copmove \robst{y,x_6}{x_1}\copwin$. The robber is trapped for every initial placement.
\end{proof}

\medskip
\section{Acknowledgments}
This work was undertaken at the Institute for Mathematics and its Applications: first as part of its Summer 2010 Special Program on Interdisciplinary Research for Undergraduates, and subsequently while the first two authors participated in  the IMA 2011-2012 Special Year on the Mathematics of Information. We are grateful to the IMA for its generous support. We also thank Volkan Isler and Vishal Saraswat for helpful conversations.

\bibliography{cop-bib}

\end{document}

%% file: fig-4-cycle-cases.tex
\path(0,0) coordinate (X3);
\path(1,1) coordinate (X2);
\path(1,-1) coordinate (X4);
\path(2,0) coordinate (X1);
\path(1,0) coordinate (X5);

\foreach \i in {1,2,3,4,5}
{
\draw[fill] (X\i) circle (2pt);
}

\node at (2.25,-.25) {$x_1$};
\node[above] at (X2) {$x_2$};
\node at (-.2,-.25) {$x_3$};
\node[below] at (X4) {$x_4$};
\node at (.75,-.25) {$x_5$};

\draw (X1) -- (X2) -- (X3) -- (X4) -- (X1);

%% file: fig-4-cycle2.tex

\path(0,0) coordinate (X3);
\path(1,1) coordinate (X2);
\path(1,-1) coordinate (X4);
\path(2,0) coordinate (X1);
\path(1,0) coordinate (X5);

\foreach \i in {1,2,3,4}
{
\draw[fill] (X\i) circle (2pt);
}

\node[right] at (X1) {$x_1$};
\node at (1.25,1.25) {$x_2$};
\node[left] at (X3) {$x_3$};
\node at (1.33,-1.15) {$x_4$};
\node at (1.33,-.25) {$x_5$};

\draw (X1) -- (X2) -- (X3) -- (X4) -- (X1);

\draw[dashed] (X2) -- (X5) -- (X4);
\draw[dashed] (X1) -- (X5);

\draw[fill=white] (X5) circle (2pt);


\path(1,3.5) coordinate (U);
\path(.25,2.5) coordinate (Y);

\draw (-.5, 2.6) -- (U) -- (2.5, 2.6);

\draw[fill=white] (1,2.5) ellipse (1.5 and .5);

\draw[fill] (U) circle (2pt);
\draw[fill] (Y) circle (2pt);

\node at (.75,3.65) {$u$};
\node at (.5,2.75) {$y$};